\documentclass[11pt]{amsart}

\usepackage[margin=1in]{geometry}
\usepackage{amsmath, amssymb, amsthm}
\usepackage{graphicx}
\usepackage{hyperref}
\usepackage{physics}
\usepackage{setspace}

\title{A Short Note on the SWAP Fidelity }

\author{Reza Rajaei}
\address{University of Oklahoma, Norman, OK}
\email{reza.rajaei-1@ou.edu}

\date{}

\newtheorem{theorem}{Theorem}

\begin{document}

\begin{abstract}
In this short note, we present a proof of a relatively expected result in a convex-optimization problem. In fact, we will show that a matrix at which the maximum of the SWAP fidelity is achieved has determinant zero.  
\end{abstract}

\maketitle

\section*{Introduction}
In \cite{MR4601184}, for any $N $, the SWAP fidelity $F_S(\rho^A,\rho^B)$, a function on $\Omega_N \times \Omega_N$, is defined to be:

$$\max_{\rho^{AB} \in \Gamma^Q(\rho^A,\rho^B)} (\Tr(S\rho^{AB})),$$

where 
$$\Gamma^Q(\rho^A,\rho^B)= \big\{  \rho^{AB} \in \Omega_{N^2} | \Tr_A (\rho^{AB})=\rho^B, \Tr_B (\rho^{AB})=\rho^A \big\},$$

and the SWAP operator $S$ is the linear operator on $\mathbb C^N \otimes \mathbb C^N$  such that $S(\ket{x} \otimes \ket{y})=\ket{y} \otimes \ket{x}.$
\singlespacing

Since the feasible set $\Gamma^Q(\rho^A,\rho^B)$ is a non-empty convex compact set, a maximizer exists and is expected to be on the boundary of the feasible set. Thus, it is reasonable to expect that any maximizer is singular, which we show is indeed the case.  
\singlespacing
As it is proved in \cite{MR4601184} that $F_S(\rho^A,\rho^B)=F_S(\rho^B,\rho^A)$, we will not be concerned with the order of partial traces in this note.

\section*{A Few Observations}
First, we recall that the partial traces of a matrix $A$ in $\mathbb C^{N \times N} \otimes \mathbb C^{N \times N}$ can be computed explicitly. In fact, assuming:

$$A=
\begin{pmatrix}
X_{11} & X_{12} & \cdots & X_{1N} \\
X_{21} & X_{22} & \cdots & X_{2N} \\
\vdots & \vdots & \ddots & \vdots \\
X_{N1} & X_{N2} & \cdots & X_{NN}

\end{pmatrix},$$

where each $X_{ij}$, $1\leq i,j \leq N$, is an $N\times N$ matrix, the partial traces of $A$ are

$$\begin{pmatrix}
\Tr(X_{11}) & \Tr(X_{12}) & \cdots & \Tr(X_{1N}) \\
\Tr(X_{21}) & \Tr(X_{22}) & \cdots & \Tr(X_{2N}) \\
\vdots & \vdots & \ddots & \vdots \\
\Tr(X_{N1}) & \Tr(X_{N2}) & \cdots & \Tr(X_{NN})

\end{pmatrix} $$

and $\sum_{i=1}^{N} X_{ii}$.
\singlespacing
Next, we compute the SWAP operator. Assuming $x=(x_0, ..., x_{N-1})^T$ and $y=(y_0, ..., y_{N-1})^T$ are in $\mathbb{C}^N$, the position of each entry of $\ket{x} \otimes \ket{y}$, when viewed as a vector, can be identified with $aN+b+1$ for $0\leq a,b\leq N-1$. For example, the first entry is identified with $(0)N+0+1$ and the last with $(N-1)N+(N-1)+1$. Now, note that the $(aN+b+1)$-th entry of $\ket{y} \otimes \ket{x}$ is $y_ax_b$. Hence, the $(bN+a+1)$-th entry of the $(aN+b+1)$-th row of $S$ is 1 and the remaining entries of the same row are zero. For example,  in the case $N=2$, the matrix representation of the SWAP operator $S$ is as follows:

$$S= 
\begin{pmatrix}
1 & 0 & 0 & 0\\
0 & 0 & 1 & 0\\
0 & 1 & 0 & 0\\
0 & 0 & 0 & 1
\end{pmatrix}. $$

At this point, we are able to compute $\Tr(SA)$ for an arbitrary matrix $A$ as described earlier. Indeed, the $(aN+b+1)$-th element on the diagonal of $SA$ is the inner product of the $(aN+b+1)$-th row of $S$ and  $(aN+b+1)$-th column of $A$. However, the only non-zero entry, which is $1$, of the $(aN+b+1)$-th row of $S$ is the $(bN+a+1)$-th entry. Therefore, the $(aN+b+1)$-th element on the diagonal of $SA$ equals the $(a+1)(b+1)$-th entry of the matrix $X_{(b+1)(a+1)}$. To visualize this argument, in the case $N=2$, we can directly verify that:

$$\Tr(
\begin{pmatrix}
1 & 0 & 0 & 0\\
0 & 0 & 1 & 0\\
0 & 1 & 0 & 0\\
0 & 0 & 0 & 1
\end{pmatrix} 
\begin{pmatrix}
X_{11} & X_{12} \\
X_{21} & X_{22}
\end{pmatrix})\ =[X_{11}]_{11}+[X_{12}]_{21}+[X_{21}]_{12}+[X_{22}]_{22},$$

where $[X]_{ij}$ denotes the $ij$-th entry of the matrix $X$.

\singlespacing

Therefore, we can equivalently use the following characterization of the SWAP fidelity: 

$$F_S(\rho^A, \rho^B)=\max\bigg\{ \sum_{1 \leq i,j \leq N} [X_{ij}]_{ji} \bigg\},$$

where the maximum is taken over all $N\times N$ matrices $X_{11}, \dotsc, X_{NN}$ such that:

i) $\rho^A=\sum_{i=1}^{N} X_{ii},$
\singlespacing
\singlespacing
ii)$\rho^B=
\begin{pmatrix}
\Tr(X_{11}) & \Tr(X_{12}) & \cdots & \Tr(X_{1N}) \\
\Tr(X_{21}) & \Tr(X_{22}) & \cdots & \Tr(X_{2N}) \\
\vdots & \vdots & \ddots & \vdots \\
\Tr(X_{N1}) & \Tr(X_{N2}) & \cdots & \Tr(X_{NN})

\end{pmatrix},$

\singlespacing
\singlespacing
iii) $
\begin{pmatrix}
X_{11} & X_{12} & \cdots & X_{1N} \\
X_{21} & X_{22} & \cdots & X_{2N} \\
\vdots & \vdots & \ddots & \vdots \\
X_{N1} & X_{N2} & \cdots & X_{NN}

\end{pmatrix}$ is an $N^2 \times N^2$ positive semidefinite  matrix. 

\section*{Main Theorem}

Using the characterization above, we state and prove our main theorem. 
\singlespacing

\begin{theorem}
    A matrix at which the maximum value of the SWAP fidelity $F_S(\rho^A, \rho^B)$ is achieved has determinant zero. 
\end{theorem}

\begin{proof}
    Let us assume 

$$A=
\begin{pmatrix}
X_{11} & X_{12} & \cdots & X_{1N} \\
X_{21} & X_{22} & \cdots & X_{2N} \\
\vdots & \vdots & \ddots & \vdots \\
X_{N1} & X_{N2} & \cdots & X_{NN}
\end{pmatrix}$$

is a matrix at which the maximum is achieved and $A$ has  non-zero determinant. Then, all the eigenvalues of $A$ are strictly greater than zero. 

Next, consider the self-adjoint $N^2 \times N^2$ matrix $J$ such that all its entries are zero except for the $(N)(N^2-N+1)$-th and $(N^2-N+1)(N)$-th entries that are $1$.  Observe that these two entries correspond to $[X_{1N}]_{N1}$ and $[X_{N1}]_{1N}$ in $A$.
On the other hand, eigenvalues are roots of the characteristic polynomial, which depends on the entries of the matrix, so for some $\epsilon >0$ small enough, the eigenvalues of $A+ \epsilon J$ are all strictly positive. Moreover, partial traces of $A+ \epsilon J$ remain equal to $\rho^A$ and $\rho^B$.

However, 

$$\Tr(S(A+ \epsilon J))=\sum_{1 \leq i,j \leq N} [X_{ij}]_{ji} + 2\epsilon > \sum_{1 \leq i,j \leq N} [X_{ij}]_{ji}=\Tr(SA),$$

which contradicts the fact that $A$ was a maximizer. For example, in the case $N=2$, $J$ is:

$$
\begin{pmatrix}
0 & 0 & 0 & 0\\
0 & 0 & 1 & 0\\
0 & 1 & 0 & 0\\
0 & 0 & 0 & 0
\end{pmatrix}, $$

and:

$$\Tr(
\begin{pmatrix}
1 & 0 & 0 & 0\\
0 & 0 & 1 & 0\\
0 & 1 & 0 & 0\\
0 & 0 & 0 & 1
\end{pmatrix} 
\begin{pmatrix}
[X_{11}]_{11} & [X_{11}]_{12} & [X_{12}]_{11} & [X_{12}]_{12}\\
[X_{11}]_{21} & [X_{11}]_{22} & [X_{12}]_{21}+\epsilon & [X_{12}]_{22}\\
[X_{21}]_{11} & [X_{21}]_{12}+\epsilon & [X_{22}]_{11} & [X_{22}]_{12}\\
[X_{21}]_{21} & [X_{21}]_{22} & [X_{22}]_{21} & [X_{22}]_{22}
\end{pmatrix})= \sum_{1 \leq i,j \leq 2} [X_{ij}]_{ji} + 2\epsilon.$$

This completes the proof. 
\end{proof}

\section*{Acknowledgments}

The author thanks Prof. J.A. Ch\'avez-Dom\'inguez for useful discussions and guidance.

\bibliography{references}

@article {MR4601184,
    AUTHOR = {Cole, Sam and Eckstein, Micha\l{} and Friedland, Shmuel and
              \.Zyczkowski, Karol},
     TITLE = {On quantum optimal transport},
   JOURNAL = {Math. Phys. Anal. Geom.},
  FJOURNAL = {Mathematical Physics, Analysis and Geometry. An International
              Journal Devoted to the Theory and Applications of Analysis and
              Geometry to Physics},
    VOLUME = {26},
      YEAR = {2023},
    NUMBER = {2},
     PAGES = {Paper No. 14, 67},
      ISSN = {1385-0172,1572-9656},
   MRCLASS = {81P40 (15A69 49Q22 82C70 90C22)},
  MRNUMBER = {4601184},
MRREVIEWER = {Nguyen\ Lam},
       DOI = {10.1007/s11040-023-09456-7},
       URL = {https://doi.org/10.1007/s11040-023-09456-7},
}
\bibliographystyle{amsalpha}




\end{document}